\documentclass[a4paper,12pt]{amsart}

\usepackage{amssymb}
\usepackage[utf8]{inputenc}
\usepackage[T1]{fontenc}
\usepackage{textcomp}
\usepackage{tikz}
\usepackage{enumitem}

\setenumerate[1]{label=(\roman{*}), ref=(\roman{*})}
\setenumerate[2]{label=(\alph{*}), ref=(\alph{*})}

\DeclareMathOperator\conv{conv}
\DeclareMathOperator\supp{supp}
\DeclareMathOperator\sgn{sgn}
\DeclareMathOperator\ext{ext}
\DeclareMathOperator\spann{span}
\newcommand{\bt}{\mathfrak{B}}
\newcommand{\mbt}{\mathfrak{M}}
\newcommand{\acF}{\mathfrak{F}} %antichain F
\newcommand\N{\mathbb N}

\newcommand{\eps}{\varepsilon}
\newcommand{\clonv}{\overline{\mbox{conv}}}
\newcommand{\ha}{h_{\mathcal{A},1}}
\newcommand{\A}{\mathcal{A}}

\newcommand{\mx}{M(x)}
\newcommand{\minfx}{M^\infty(x)}
\newcommand{\mend}{M^{\mathcal{F}}}
\newcommand{\norm}[1]{\|#1\|}

\newtheorem{thm}{Theorem}[section]
\newtheorem{prop}[thm]{Proposition}
\newtheorem{lem}[thm]{Lemma}
\newtheorem{cor}[thm]{Corollary}

\theoremstyle{definition}
\newtheorem{defn}[thm]{Definition}

\newtheorem{quest}[thm]{Question}

\theoremstyle{remark}
\newtheorem{fact}[thm]{Fact}

\newtheorem{rem}[thm]{Remark}

\title[Daugavet-, delta-points and uncond.\ bases]
{Daugavet- and delta-points in Banach spaces with unconditional bases}

\author[T.~A.~Abrahamsen]{Trond A.~Abrahamsen}
\address[T.~A.~Abrahamsen]{Department of Mathematics, University of
  Agder, Postboks 422, 4604 Kristiansand, Norway.}
\email{trond.a.abrahamsen@uia.no}
\urladdr{http://home.uia.no/trondaa/index.php3}

\author[V.~Lima]{Vegard Lima}
\address[V.~Lima]{Department of Engineering Sciences, University of Agder,
Postboks 509, 4898 Grimstad, Norway.}
\email{Vegard.Lima@uia.no}

\author[A.~Martiny]{Andr\'e Martiny}
\address[A.~Martiny]{Department of Mathematics, University of
  Agder, Postboks 422, 4604 Kristiansand, Norway.}
\email{andre.martiny@uia.no}

\author[S.~Troyanski]{Stanimir Troyanski}
\address[S.~Troyanski]{Institute of Mathematics and Informatics,
  Bulgarian Academy of Science, bl.8, acad. G.~Bonchev str.~1113 Sofia, Bulgaria and
  Departamento de Matem\'aticas, Universidad de Murcia, Campus de
  Espinardo, 30100 Espinardo (Murcia), Spain}
\email{stroya@um.es}



\thanks{The fourth named author was supported by MTM2017-86182-P
  (AEI/FEDER, UE), and Bulgarian National Scientific Fund, Grant,
  KP–06–H22/4, 04.12.2018.}

\begin{document}

\begin{abstract} We study the existence of Daugavet- and delta-points
  in the unit sphere of Banach spaces with a $1$-unconditional basis. A
  norm one element $x$ in a Banach space is a Daugavet-point
  (resp.\ delta-point) if every element in the unit ball
  (resp.\ $x$ itself) is in the closed convex hull of unit ball
  elements that are almost at distance $2$ from $x$. A Banach
  space has the Daugavet property (resp.\ diametral local diameter two
  property) if and only if every norm one element is a Daugavet-point
  (resp.\ delta-point). It is well-known that a Banach space with the
  Daugavet property does not have an unconditional basis.
  Similarly spaces with the diametral local diameter two property
  do not have an unconditional basis with suppression unconditional
  constant strictly less than $2$.

  We show that no Banach space with a subsymmetric basis can
  have delta-points. In contrast we construct a Banach space
  with a $1$-unconditional basis with delta-points, but with no
  Daugavet-points, and a Banach space with a $1$-unconditional basis
  with a unit ball in which the Daugavet-points are weakly dense.
\end{abstract}

\maketitle

\section{Introduction}
\label{sec:intro}
Let $X$ be a Banach space with unit ball $B_X$, unit sphere $S_X,$ and
topological dual $X^*$.
For $x \in S_X$ and $\eps > 0$ let $\Delta_\eps(x) = \{y \in B_X:
\|x -y\| \ge 2 -\eps\}$. We say that $X$ has the
\begin{enumerate}
  \item \label{item:daug}\emph{Daugavet property} if for every $x
    \in S_X$ and every $\eps > 0$ we have $B_X =\clonv\Delta_\eps(x)$;
  \item \label{item:delta}\emph{diametral local diameter two property}
    if for every $x
    \in S_X$ and every $\eps > 0$ we have $x \in \clonv\Delta_\eps(x)$.
\end{enumerate}

In \cite[Corollary~2.3]{Kadets} Kadets proved that any
Banach space with the Daugavet property fails to have an unconditional
basis (see also
\cite[Proposition~3.1]{MR1856978}). These arguments are probably the
easiest known proofs of the absence
of unconditional bases in the classical Banach spaces $C[0,1]$ and
$L_1[0,1]$.
The diametral local diameter two property
was named and studied in \cite{BGLPRZ-diametral},
but it was first introduced in \cite{zbMATH02168839}
under the name \emph{space with bad projections}.
(See the references in \cite{zbMATH02168839} for previous
unnamed appearances of this property.)
Using the characterizations in \cite{zbMATH02168839}
we see that if a Banach space
with the diametral local diameter two property
has an unconditional basis, then
the unconditional suppression basis constant is at least $2$.
But note that
we do not know of any Banach space with an unconditional basis
and the diametral local diameter two property.

In the present paper we study pointwise versions of the
Daugavet property and the diametral local diameter two property
in spaces with $1$-unconditional bases.

\begin{defn}
  \label{defn:daug-delta-points}
  Let $X$ be a Banach space and let $x \in S_X$. We say that $x$ is
  \begin{enumerate}
  \item a \emph{Daugavet-point} if for every $\eps > 0$ we have
    $B_X = \clonv\Delta_\eps(x)$;
  \item a \emph{delta-point} if for every $\eps > 0$ we have
    $x \in \clonv\Delta_\eps(x)$.
  \end{enumerate}
\end{defn}

Daugavet-points and delta-points were introduced in
\cite{AHLP}.
For the spaces $L_1(\mu)$, for preduals of
such spaces, and for M{\"u}ntz spaces these notions are the same
\cite[Theorems~3.1, 3.7, and 3.13]{AHLP}.
However, $C[0,1] \oplus_2 C[0,1]$ is an example of a space with the
diametral local diameter two property, but
with no Daugavet-points \cite[Example~4.7]{AHLP}.
Stability results for Daugavet- and delta-points
in absolute sums of Banach spaces
was further studied in \cite{haller2020daugavet}.

In Section~\ref{sec:1unc-nodelta} we consider
Banach spaces with $1$-unconditional bases and study a
family of subsets of the support of a vector $x$.
We find properties of these subsets that are intimately
linked to $x$ not being a delta-point.
Quite general results are obtained in this direction.
We apply these results to show that Banach spaces
with subsymmetric bases (these include separable
Lorentz and Orlicz sequence spaces)
always fail to contain delta-points.

In Section~\ref{sec:binarytree} we
construct a Banach space with a $1$-unconditional basis
which contains a delta-point, but contain no Daugavet-points. The example
is a Banach space of the type $\ha$ generated by an adequate family of
subsets of a binary tree.
The norm of the space is the supremum of the $\ell_1$-sum
of branches in the binary tree.

In Section~\ref{sec:modbinarytree} we modify slightly the
binary tree from Section~\ref{sec:binarytree} and
the associated adequate family,
to obtain an $\ha$ space with some remarkable properties:
It has Daugavet-points;
the Daugavet-points are even weakly dense in the unit ball;
the diameter of every slice of the unit ball is two,
but is has relatively weakly open subsets of the unit ball
of arbitrary small diameter.

Finally, let us also remark that the examples in both
Section~\ref{sec:binarytree} and Section~\ref{sec:modbinarytree}
contain isometric copies of $c_0$ and $\ell_1$.
Both the $\ell_1$-ness of the branches and $c_0$-ness
of antichains in the binary tree play an important
role in our construction of Daugavet- and delta-points
in these spaces (see
e.g. Theorems~\ref{thm:bt-har-delta-ikke-daugavet}
and \ref{thm:mbt_kar_daugpkt}, and
Corollary~\ref{cor:konkret-betingelse-for-daugavetpkt}).

\section{$1$-unconditional bases and the sets $M(x)$}
\label{sec:1unc-nodelta}
The main goal of this section is to prove that Banach spaces with a
subsymmetric basis fail to have delta-points. Before we start this
mission, let us point out some results and
concepts that we will need. First some characterizations of
Daugavet- and delta-points that we will frequently use throughout the paper.

Recall that a \emph{slice} of the unit ball $B_X$ of a Banach space
$X$ is a subset of the
form
\begin{align*}
  S(x^*,\eps) = \{x \in B_X: x^*(x) > \|x^*\| - \eps\},
\end{align*}
where $x^* \in X^*$ and $\eps > 0$.

\begin{prop}{\cite[Lemma~2.3]{AHLP}}\label{prop:Daugavet_point_crit}
       Let $X$ be a Banach space and $x \in S_X$.
       The following assertions are equivalent:
       \begin{enumerate}
              \item\label{aaaa}
              $x$ is a Daugavet-point;
              \item\label{bbbb}
              for every slice $S$ of $B_X$ and for every $\varepsilon>0$
              there exists $y\in S$ such that $\|x - y\| \geq 2-\varepsilon$.
       \end{enumerate}
\end{prop}

\begin{prop}{\cite[Lemma~2.2]{AHLP}}\label{prop:Delta_point_crit}
       Let $X$ be a Banach space and $x \in S_X$.
       The following assertions are equivalent:
       \begin{enumerate}
              \item\label{aaa}
              $x$ is a delta-point;
              \item\label{bbb}
              for every slice $S$ of $B_X$ with $x\in S$ and for every
              $\varepsilon > 0$ there exists $y \in S_X$
              such that $\|x - y\| \geq 2 - \varepsilon$.
       \end{enumerate}
\end{prop}

Let $X$ be a Banach space. Recall that a Schauder basis $(e_i)_{i \in \N}$ of
$X$ is called \emph{unconditional} if for every $x \in X$ its
expansion $x = \sum_{i \in \N}  x_i e_i$ converges unconditionally.
If, moreover,
$\|\sum_{i \in \N} \theta_i x_i e_i\| = \|\sum_{i \in \N} x_i e_i\|$
for any $x = \sum_{i \in \N}  x_i e_i \in X$
and any sequence of signs $(\theta_i)_{i \in \N}$,
then $(e_i)_{i \in \N}$ is called $1$-unconditional.
A Schauder basis is called \emph{subsymmetric},
or \emph{$1$-subsymmetric}, if it is
unconditional and
$\|\sum_{i \in \N} \theta_ix_ie_{k_i}\| = \|\sum_{i \in \N} x_i e_i\|$ for
any $x = \sum_{i \in \N}  x_i e_i \in X$,
any sequence of signs $(\theta_i)_{i \in \N}$, and
any infinite increasing sequence of naturals $(k_i)_{i \in \N}$.
Trivially a subsymmetric basis is $1$-unconditional.
In the following we will assume that the basis $(e_i)_{i \in \N}$
is normalized, i.e. $\|e_i\| = 1$ for all $i \in \N$.
With $(e^*_i)_{i \in \N}$ we denote the conjugate in $X^*$
to the basis $(e_i)_{i \in \N}$.
Clearly $(e^*_i)_{i \in \N}$ is a $1$-unconditional basic sequence
whenever $(e_i)_{i \in \N}$ is.
When studying Daugavet-points or delta-points in
a Banach space $X$ with $1$-unconditional basis $(e_i)_{i \in \N}$
we can restrict our investigation
to the positive cone $K_X$ generated by the basis, where
\begin{equation*}
  K_X = \left\{x = \sum_{i \in \N} x_i e_i : x_i \ge 0\right\}
  = \{x \in X : e^*_i(x) \ge 0\}.
\end{equation*}
The reason for this is that for every sequence of signs
$\theta = (\theta_i)_{i \in \N}$ the operator
$T_\theta : X \to X$ defined by
$T_\theta(\sum_{i \in \N} x_i e_i) = \sum_{i \in \N} \theta_i x_i e_i$
is a linear isometry.
Hence $x = \sum_{i \in \N} x_i e_i$ is a Daugavet-point (resp. delta-point)
if and only if $|x| = \sum_{i \in \N} |x_i| e_i$ is.

The following result is well-known.
\begin{prop}
  \label{prop:1-unc-mult-bounded-seq}
  Let $X$ be a Banach space with a $1$-unconditional basis
  $(e_i)_{i \in \N}$.
  If $\sum_{i \in \N} b_ie_i$ is convergent and $|a_i| \le |b_i|$ for
  all $i$, then $\sum_{i \in \N} a_ie_i$ is convergent and
  \begin{align*}
    \big\|\sum_{i \in \N} a_ie_i\big\|
    \le \big\|\sum_{i \in \N} b_ie_i\big\|.
  \end{align*}
  Moreover $\|P_A\| = 1$ where, for $A \subset \N$,
  $P_A$ is the projection defined by
  \begin{equation*}
    P_A(\sum_{i \in \N} x_i e_i) = \sum_{i \in A} x_i e_i.
  \end{equation*}
\end{prop}

From this we immediately get a fact that will
be applied several times throughout the paper.

\begin{fact}
  \label{fact:1-unc}
  Let $X$ be a Banach space with a $1$-unconditional basis
  $(e_i)_{i \in \N}$
  and let $x,y \in X$ and $E \subset \N$.
  Then the following holds.
  \begin{itemize}
  \item
    If $|x_i| \le |y_i|$ and $\sgn x_i = \sgn y_i$
    for all $i \in E$,
    then $\norm{y - P_Ex} \le \norm{y}$.
  \end{itemize}
\end{fact}

The upshot of Fact~\ref{fact:1-unc} is that it can be
used to find an
upper bound for the distance between $x \in S_X$ and
elements in a given subset of the unit ball.
Indeed, suppose we can find $E \subseteq \N$,
$\eta > 0$ and a subset $S$ of the unit ball
such that $\|x - P_E x\| < 1 -\eta$ and
the assumption in Fact~\ref{fact:1-unc}
holds for any $y \in S$.
Then
\[
\norm{x - y} \le \norm{x - P_Ex} + \norm{y - P_Ex} < 2 - \eta.
\]
If such a set $S$ is a slice (resp.\ a slice containing $x$), then $x$ cannot be a
Daugavet-point (resp.\ delta-point). We will see in
Theorem~\ref{thm:subsym-no-delta} that any unit sphere element in a
space with a subsymmetric basis, is contained in a slice of the above type.
Our tool to investigate the existence of slices of
this type in a Banach space with a $1$-unconditional basis, are
certain families of subsets of the support of the elements in the space.

\begin{rem}
  If only the moreover part of
  Proposition~\ref{prop:1-unc-mult-bounded-seq} holds,
  then the basis is called $1$-suppression unconditional.
  In this case the conclusion of
  Proposition~\ref{prop:1-unc-mult-bounded-seq}
  still holds if $\sgn a_i = \sgn b_i$, for all $i$.
  This is all that is needed in Fact~\ref{fact:1-unc}.
  Similarly, one can check that all the results
  about $1$-unconditional bases in
  the rest of this section also holds for a Banach space
  $X$ with a $1$-suppression unconditional basis.
\end{rem}

\begin{defn}
       \label{defn:cx}
       For any Banach space $X$ with $1$-unconditional basis
       $(e_i)_{i \in \N}$ and for
       $x\in X$, define
       \[
       \mx :=\left\{A\subseteq \N : \left\|P_Ax\right\|= \left\| x\right\|,
         \left\|P_Ax-x_ie_i\right\|<\left\| x\right\|,
         \ \mbox{for all}\ i\in A \right\},
       \]
       \[
       \mend(x) := \left\{A\in \mx : \left|A \right|<\infty \right\},
       \]
       and
       \[
       \minfx := \left\{A\in \mx : |A|=\infty \right\}.
       \]
\end{defn}

We can think of $\mx$ as a collection of minimal ``norm-giving'' subsets
of the support of $x$. If for example $X
= c_0$ and $x \in c_0$, then $\mx = \{\{i\}: |x_i| = \norm{x}\}$
while if $X =\ell_p$, $1 \le p < \infty$ and $x \in X$, then $\mx = \{\supp(x)\}$.

Our first observation about the families $\mx$ is that they are always non-empty.

\begin{lem}\label{lem:overkill}
  Let $X$ be a Banach space with $1$-unconditional basis $(e_i)_{i \in \N}$.
  Then $\mx \neq \emptyset$ for all $x \in X$.
\end{lem}

\begin{proof}
  Let $x\in X$.
  Either $A_0 := \supp(x) \in \mx$ or
  there exists a smallest $n_1\in A_0$ such
  that if we define $A_1 = A_0 \setminus \{n_1\}$,
  then
  $\left\|P_{A_1}  x \right\| = \left\|x\right\|$ and
  \begin{equation*}
    \left\|P_{A_1}  x - x_{j}e_j\right\|< \left\|x\right\|
    \ \mbox{for all}\ j
    \in A_0 \cap \{1,\ldots,n_1-1\}.
  \end{equation*}
  Suppose we have found $n_1 < \cdots < n_{k-1}$
  such that $A_{k-1}  = A_{k-2} \setminus \{n_{k-1}\}$
  satisfies $\|P_{A_{k-1} }  x\| = \|x\|$ and
  $\|P_{A_{k-1} }  x - x_j e_j\| < \|x\|$ for all
  $j \in A_{k-1}  \cap \{1,\ldots,n_{k-1}-1\}$.
  Then either $A_{k-1}  \in \mx$ or there exists
  a smallest integer $n_k$ greater than $n_{k-1}$ such that
  $A_k = A_{k-1}(x) \setminus \{n_k\}$ satisfies
  $\left\|P_{A_k}  x\right\| = \left\|x\right\|$
  and
  \begin{equation*}
    \left\|P_{A_k}  x - x_j e_j \right\| < \left\|x\right\|
    \ \mbox{for all}\ j
    \in A_k \cap \{1,\ldots,n_k-1\}.
  \end{equation*}
  Either this process terminates and $A_k\in \mx$, or we
  get a set $N = \left\{n_i \right\} _{i=1}^\infty$.
  Let $A = \bigcap_k A_k = \supp(x)\setminus N$ and note that
  $\left\|P_A x\right\| = \left\|x\right\|$.
  If $j \in A$, find $k$ such that $j < n_k$, then
  by $1$-unconditionality
  \begin{equation*}
    \|P_A x - x_j e_j\|
    \le
    \| P_{A_k}x - x_j e_j \|
    < \|x\|
  \end{equation*}
  and $A \in \mx$.
\end{proof}

Our next goal is to prove that certain classes of subsets of
$\mend(x)$ and $\minfx$ are finite (see Lemma~\ref{lem:k_nendelig}
below). We will use the next result as a stepping stone. In the proof,
and throughout the paper, we will assume that the sets
$A = \left\{a_1, a_2, \ldots \right\} \in \mx$ are ordered
so that $a_1 < a_2 < \cdots < a_n <\cdots$, and we will
use $A(n)$ to denote the set $\{a_1,\ldots,a_n\}$.

\begin{lem}\label{lem:i-frste-c-infty}
  Let $X$ be a Banach space with $1$-unconditional basis
  $(e_i)_{i \in \N}$. If $x\in X$,
  then for every $n \in \N$,
  \begin{enumerate}
  \item\label{item:lem-ci-1}
    $\left|\left\{A(n) : A\in \mx,
        \ \left|A \right|>n  \right\} \right|< \infty$;
  \item\label{item:lem-ci-2}
    $\left|\left\{A\in \mx
        : \left|A \right|\leq n \right\} \right| < \infty$.
  \end{enumerate}
  In particular,
  $\left|\bigcup_{D\in \minfx}\left\{D(n) \right\} \right|<\infty$.
\end{lem}
\begin{proof}
  Let us prove \ref{item:lem-ci-1} inductively.
  For $k \in \N$, let $R_k = I - P_{N_k}$,
  where $N_k = \{1,\ldots,k\}$.
  For $n = 1$ the result follows from
  $\left\|R_k x\right\|  \rightarrow 0$.

  Now assume that
  $\left|\left\{A(n-1):
      A\in \mx,\ \left|A \right| > n-1 \right\}
  \right|< \infty$,
  and let
  $s_{n-1} :=
  \max\left\{\left\|P_{A\left(n-1 \right)}x\right\| : A\in
    \mx, \left|A \right|>n-1  \right\}< \left\|x\right\|$.
  Find $k \in \N$  such that
  $\left\|R_k x\right\| < \left\|x\right\| - s_{n-1}$.
  Then by the triangle inequality, it follows that
  $\max A(n)\leq k$ for all $A\in \mx$
  with $\left|A \right|> n$.

  For \ref{item:lem-ci-2},
  let $A\in \mx$ with $\left|A \right|= n$.
  Then $\left\|P_{A\left(n-1 \right)}x\right\| \leq s_{n-1}$,
  and thus $\max A \leq k$, where as above
  $k\in \N$ is such that $\|R_k x\| < \|x\|- s_{n-1}$.
\end{proof}

In order to find the sets $E \subseteq \N$ mentioned
in the remarks following Fact~\ref{fact:1-unc}
we need the following families of subsets of $\mx$.

\begin{defn}
       \label{defn:fn-gn-en}
       Let $X$ have $1$-unconditional basis
       $(e_i)_{i \in \N}$. Let $x\in S_X$ and define
       \begin{align*}
       \mathcal{F}_n(x)
       &:= \left\{A\in \mend(x) : A \cap D(n) \neq D(n),
       \ \mbox{for all}\ D\in \minfx\right\},\\
       \mathcal{G}_n(x)
       &:= \mathcal{F}_n(x) \cup \bigcup_{D \in \minfx} \{D(n)\},\\
       \mathcal{E}_n(x)
       &:= \left\{E\subset \!\! \bigcup_{A\in
              \mathcal{G}_n} \!\!A : E\cap A \neq
       \emptyset, \ \mbox{for all}\ A\in \mathcal{G}_n
       \right\}.
       \end{align*}
       If it is clear from the context what element $x$ we are considering,
       we will simply denote these sets by $\mathcal{F}_n, \mathcal{G}_n,$
       and $\mathcal{E}_n$.
\end{defn}

It is pertinent with a couple of comments about these families of sets. Trivially, if
$\minfx = \emptyset$, then $\mathcal{G}_n =
\mathcal{F}_n = \mx$ for all
$n\in \N$. We can think of the elements of $\mathcal{E}_n$ as
essential for the norm of $x$, i.e.
$\norm{x - P_Ex} < \norm{x}$ for all $E \in \mathcal{E}_n$.
According to Lemma~\ref{lem:G_N-minker} below the drop in norm is also
uniformly bounded away from $0$.
The main reason for this is that $\mathcal{F}_n$ and $\mathcal{E}_n$
are finite for all $n \in \N$. We will
prove this now.

\begin{lem}\label{lem:k_nendelig}
  Let $X$ have $1$-unconditional basis
  $(e_i)_{i \in \N}$. If $x\in S_X$,
  then for all $n\in \N$,
  \begin{enumerate}
  \item\label{item:k_n-1}
    $|\mathcal{F}_n|<\infty$;
  \item\label{item:k_n-2}
    $|\mathcal{E}_n|<\infty$.
  \end{enumerate}
  In particular, if $\minfx = \emptyset$, then $|\mx|<\infty$.
\end{lem}

\begin{proof}
  \ref{item:k_n-1}.
  There exists $N \in \N$ such that
  $\max_{D\in \minfx}D(n) \le N$
  by Lemma~\ref{lem:i-frste-c-infty}.

       Assume for contradiction that $|\mathcal{F}_n| = \infty$.
       Then there exists a sequence $(A_k) \subset \mathcal{F}_n$
       such that $|A_k| \ge k$.
       By compactness of $\{0,1\}^{\N}$ and passing
       to a subsequence if necessary,
       we may assume that
       $A_k \rightarrow A \in \N$ pointwise and
       $A \cap \{1,\ldots,N\} = A_k \cap \{1,\ldots,N\}$
       for all $k$.
       In particular $\left\|P_A x\right\| = 1$.
       By Lemma~\ref{lem:overkill}, there exists $B\subseteq A$,
       such that $B\in M(P_A x) \subseteq \mx$.
       Since $A \cap \{1,\ldots,N\} = A_k \cap \{1,\ldots,N\},$
       we have $|B| < \infty$ by definition of $\mathcal{F}_n$. Since $B$
       is finite $A_k \cap B$ is eventually constant. Thus for some $k \in \N$ we have
       $B \subsetneq A_k \in \mx,$ a contradiction.

       Finally, \ref{item:k_n-2} follows from \ref{item:k_n-1}
       and Lemma~\ref{lem:i-frste-c-infty}.
\end{proof}

With the knowledge that the cardinality of $\mathcal{E}_n$ is finite
for every $n \in \N$, we now obtain the following result.

\begin{lem}\label{lem:G_N-minker}
  Let $X$ be a Banach space with $1$-unconditional basis
  $(e_i)_{i \in \N}$.
  If $x\in S_X$, then
  \begin{enumerate}
  \item\label{item:G_N-minker-1}
    $\|x - P_E x\| < 1$
    if $E \cap A \neq \emptyset$
    for all $A \in \mx$;
  \item\label{item:G_N-minker-2}
    for any $n \in \N$ there exists
    $\gamma_n > 0$ such that
    \begin{equation*}
      \max\limits_{E \in \mathcal{E}_n} \left\|x - P_E x\right\|
      = 1 - \gamma_n.
    \end{equation*}
  \end{enumerate}
\end{lem}

\begin{proof}
  \ref{item:G_N-minker-1}.
  Assume that $E \subseteq \N$ with $E \cap A \neq \emptyset$
  for all $A \in \mx$ such that $\|x - P_E x\| = 1$.
  By Lemma~\ref{lem:overkill} there exists
  $B \in M(x - P_E x)$.
  But $M(x - P_E x) \subseteq \mx$ since $\|x - P_E x\| = 1$
  and this gives us the contradiction $B \cap E = \emptyset$.

  Any $E \in \mathcal{E}_n$ satisfies
  $E \cap A \neq \emptyset$ for all $A \in \mx$
  and $\mathcal{E}_n$ is finite, so \ref{item:G_N-minker-2}
  follows from \ref{item:G_N-minker-1}.
\end{proof}

Let $X$ be a Banach space and $x \in S_X$.
If $x$ is a delta-point, then for every slice $S$
with $x \in S$, we have that $x$ is at one end
of a line segment in $S$ with length as close
to $2$ as we want.
Suppose we replace the slice $S$ with
a non-empty relatively weakly open subset $W$ of $B_X$
with $x \in W$.
If $X$ has the Daugavet property, then $x$ is at one
end of a line segment in $W$ with length as close
to $2$ as we want (\cite[Lemma~3]{MR1784413}).
Next we show that this is never the case
if $X$ has a $1$-unconditional basis.

\begin{prop}\label{prop:1-unc-no-b-points}
  Let $X$ be a Banach space with $1$-unconditional basis
  $(e_i)_{i \in \N}$.
  If $x \in S_X$, then there exist $\delta>0$ and
  a relatively weakly open subset $W$, with $x\in W$,
  such that $\sup_{y \in W} \left\|x-y\right\| < 2 - \delta$.
\end{prop}
\begin{proof}
  Assume that $x\in S_X \cap K_X$.
  Let $E = \bigcup_{A \in \mx} A(1)$.
  By Lemma~\ref{lem:G_N-minker} there exists $\gamma_1 > 0$
  such that $\max_{F \in \mathcal{E}_1} \|x - P_Fx\| = 1 - \gamma_1$.
  Let $\delta = \gamma_1/2$.

  Let $W = \left\{y\in B_X : \left|e_i^*\left(x-y \right)
    \right|<\min_{k\in E}\frac{x_k}{2}, i \in E  \right\}$. Then
  $x\in W$, and if $y\in W$, then $y_i\geq \frac{x_i}{2}>0$ for
  all $i\in E$. Thus if $y \in W$ we have
  \begin{align*}
    \{i \in \N: y_i \ge \frac{x_i}{2}\} \cap E
    = E \in \mathcal{E}_1.
  \end{align*}
  For any $y \in W$, we get that
  \begin{align*}
    \norm{x-y}
    & \le \bigg\|\frac{x}{2}\bigg\|
    + \bigg\|\frac{x}{2} -P_E\frac{x}{2}\bigg\|
    + \bigg\| P_E\frac{x}{2} - y\bigg\|
    < 2 - \delta,
  \end{align*}
  and we are done.
\end{proof}

Let us remark a fun application of the above proposition.

\begin{rem}
  Let $K$ be an infinite compact Hausdorff space.
  Then $C(K)$ does not have a $1$-unconditional
  (or a $1$-suppression unconditional) basis.

  Let $f$ be a function which attains its norm on a limit point of $K$.
  Arguing similarly as in \cite[Theorem~3.4]{AHLP}
  we may find a sequence of norm one
  functions $g_k$ with distance as close to $2$ as we want from
  $f$ that converge pointwise, and thus weakly, to $f$.
  The conclusion follows from Proposition~\ref{prop:1-unc-no-b-points}.
\end{rem}

The next result is the key ingredient in our proof that there are no
delta-points in Banach spaces with subsymmetric bases. Its proof
draws heavily upon Lemma~\ref{lem:G_N-minker}.

\begin{lem}\label{lem:condition-on-C_infty}
       Let $X$ be a Banach space with $1$-unconditional
       basis $(e_i)_{i \in \N}$ and let $x\in S_X$.
       Assume that there exists a slice $S(x^*,\delta)$,
       an $n\in \N$ and some $\eta>0$ such that
       \begin{enumerate}
              \item\label{item:hl-1}
              $x\in S(x^*,\delta)$,
              \item\label{item:hl-2}
              $y\in S(x^*,\delta)$ implies
              that
              \begin{equation*}
              \left\{i:|y_i|>\eta |x_i|  , \ \sgn y_i =  \sgn x_i \right\}
              \cap D(n) \neq \emptyset
              \end{equation*}
              for all $D \in \minfx$.
       \end{enumerate}
       Then $x$ is not a delta-point.
\end{lem}

\begin{proof}
       Assume that $x \in S_X \cap K_X$.
       Now for each $A\in \mathcal{F}_n$ find
       $x_A^* \in S_{X^*}$ such that $x_A^*(P_A x) = 1$
       with $x_A^*(e_i)=0$ for all $i\notin A$,
       and $x_A^*(e_i)>0$ for all $i\in A$.
       Let $z^* = \frac{1}{|\mathcal{F}_n|+1}\left(\sum_{A\in \mathcal{F}_n}
       x_A^* + x^*\right)$. Then $z^*\in B_{X^*}$ and
       \[
       \left\|z^*\right\|\geq z^*(x)> \frac{|\mathcal{F}_n|+1-\delta}{|\mathcal{F}_n|+1} =
       1-\frac{\delta}{|\mathcal{F}_n|+1}.
       \]
       For any $y\in
       S(z^*,\left\|z^*\right\|-1+\frac{\delta}{|\mathcal{F}_n|+1})$, we
       get that
       \begin{align*}
       1-\frac{\delta}{|\mathcal{F}_n|+1}
       & < \frac{1}{|\mathcal{F}_n|+1}\left(\sum_{A\in \mathcal{F}_n} x_A^*(y) +
       x^*(y)\right) \leq  \frac{|\mathcal{F}_n|+x^*(y)}{|\mathcal{F}_n|+1}.
       \end{align*}
       Solving for $x^*(y)$ we get that
       \[
       1-\delta<x^*(y),
       \]
       and similarly $1-\delta< x_A^*(y)$. Thus, if $0 < \eta<1-\delta$,
       \[
       \displaystyle F := \left\{ i : y_i \geq
       \eta x_i\right\} \bigcap \left(\bigcup_{E\in \mathcal{G}_n}E
       \right) \in \mathcal{E}_n.
       \]
       For any $y \in
       S(z^*,\|z^*\|-1+\frac{\delta}{|\mathcal{F}_n|+1})$
       we now get from Lemma~\ref{lem:G_N-minker} that
       \begin{align*}
       \left\|x-y\right\|
       & \leq \left\|x-\eta P_F x\right\|+\left\|\eta P_Fx - y\right\| \\
       & \leq \eta\left\|x-P_F x\right\|+\left(1-\eta\right)\left\|x\right\|+1 \\
       & \leq \eta\max_{E\in \mathcal{E}_n} \left\|x-P_E x\right\| + 2-\eta \\
       & \leq 2-\eta\gamma_n<2.
       \end{align*}
\end{proof}

If $x \in S_X$ with $\minfx = \emptyset$
in the above lemma, then any slice $S(x^*,\delta)$
containing $x$ trivially satisfies
Lemma~\ref{lem:condition-on-C_infty}~\ref{item:hl-2}.
We record this in the following proposition.

\begin{prop}\label{prop:minfx_empty_no_delta}
  Let $X$ be a Banach space with $1$-unconditional basis
  and let $x \in S_X$.
  If $\minfx = \emptyset$, then $x$ is not a delta-point.
\end{prop}

By definition of the sets $\mx$ and a convexity argument, the next
result should be clear.

\begin{lem}
       \label{lem:mx-add-norm-increase}
       Let $X$ be a Banach space with $1$-unconditional basis
       $(e_i)_{i \in \N}$. If $x \in K_X$,
       then for every $A \in \mx$
       and every $t > 0$ we have
       $\norm{P_Ax + te_i} > \norm{x}$ for all $i \in A$.
\end{lem}

Finally it is time to cash in some dividends and prove the main result of
this section.

\begin{thm}
       \label{thm:subsym-no-delta}
       If $X$ has subsymmetric basis $(e_i)_{i \in \N}$,
       then $X$ has no delta-points.
\end{thm}

\begin{proof}
  Assume $x \in S_X \cap K_X$.
  By Proposition~\ref{prop:minfx_empty_no_delta} we may
  assume that $\minfx \neq \emptyset$.
  Let $s:=\max\{ n : x_n = \max_i x_i\}$.
  We first show that $s \in A$ for all $A \in \minfx$.

  For contradiction assume that there exists
  $A = \{a_1, a_2, \ldots\} \in \minfx$
  with $s \not\in A$. Let $a_0 = 0$ and $j \in \N$ be such
  that $a_{j - 1} < s < a_j$. Let $t > 0$ such that $x_s = x_{a_j}+t$
  and let $A_s$ be $A$ with $a_j$ replaced by $s$.
  Using that $(e_i)_{i \in \N}$ is subsymmetric and
  Lemma~\ref{lem:mx-add-norm-increase} we get
  \begin{align*}
    1 \ge \norm{P_{A_s}x}
    &=
    \bigg\|
    \sum_{i \not = j} x_{a_i}e_{i} + (x_{a_j} + t) e_j
    \bigg\|\\
    &=
    \bigg\|\sum_{i \in \N} x_{a_i}e_{a_i} + te_{a_j} \bigg\|
    = \|P_Ax + t e_{a_j} \| > 1
  \end{align*}
  a contradiction.

  If we let $n = s$, then $s \in D(n)$ for all $D \in \minfx$,
  and the slice $S(e^*_s, 1 - \frac{x_s}{2})$ and $\eta = \frac{1}{2}$
  satisfies the criteria in
  Lemma~\ref{lem:condition-on-C_infty} and we are done.
\end{proof}

In the proof above we saw that if $X$ has a subsymmetric basis,
then for any $x \in S_X$ either $\minfx = \emptyset$ or
all $A \in \minfx$ has a common element.
In the case $X$ has a $1$-symmetric basis we can say a lot about the
sets $\mx$ for any given $x \in S_X$.

Recall that a Schauder basis $(e_i)_{i \in \N}$
is called \emph{$1$-symmetric} if it is
unconditional and $\|\sum_{i \in \N} \theta_ix_ie_{\pi(i)}\|
= \|\sum_{i \in \N} x_i e_i\|$ for any
$x = \sum_{i \in \N}  x_i e_i \in X,$ any sequence of signs
$(\theta_i)_{i \in \N}$, and any permutation $\pi$ of $\N$. A
$1$-symmetric basis is subsymmetric \cite[Proposition~3.a.3]{LiTz1}.

\begin{prop}
       \label{prop:1-sym-mx}
       Let $X$ be a Banach space with $1$-symmetric basis
       $(e_i)_{i \in \N}$ and let $x \in S_X$.
       \begin{enumerate}
       \item\label{item:1-sym-mx_1}
         If $\minfx \not=\emptyset,$ then $\mx = \{\supp(x)\};$
       \item\label{item:1-sym-mx_2}
         If $\minfx = \emptyset$ and $A, B \in \mx,$ then $|A| =
         |B|$ and $x$ is constant on $A \triangle B$.
       \end{enumerate}
\end{prop}

\begin{proof}
       Assume that $x \in S_X \cap K_X$.

       \ref{item:1-sym-mx_1}.
       Let $A \in \minfx$ and $x_l \in \supp(x) \setminus A$.
       Since $|A| = \infty$, there exists $k \in A$ and $t > 0$ with $x_k +t = x_l$. Using
       that $(e_i)_{i \in \N}$ is $1$-symmetric and
       Lemma~\ref{lem:mx-add-norm-increase} we get
       \begin{align*}
       1 \geq \left\|P_{A\setminus\left\{k\right\}} x + x_l
         e_l\right\| = \left\|P_{A\setminus\left\{k \right\}}x + x_l
         e_k\right\| = \left\|P_A x + t e_k\right\| > 1,
       \end{align*}
       a contradiction.

       \ref{item:1-sym-mx_2}. Suppose that $x$
       is not constant on $A \triangle B$ and let $k, l \in A \triangle B$
       with $x_k \not= x_l,$ say $k \in A, l \in B,$ and $x_k <
       x_l$. Then argue as in \ref{item:1-sym-mx_1} to get a
       contradiction, so $x$ is
       constant on $A \triangle B$. As $x$ is constant on $A\triangle
       B$, we cannot have $\left|A \right|< \left|B \right|$ since
       then a subset of $B$ would be in $\mx$ contradicting the
       definition of $\mx$.
\end{proof}

\section{A space with $1$-unconditional basis and delta-points}
\label{sec:binarytree}

In this section we will prove the following theorem.

\begin{thm}\label{thm:bt-har-delta-ikke-daugavet}
  There exists a Banach space $X_{\bt}$
  with $1$-unconditional basis,
  such that
  \begin{enumerate}
  \item\label{thmbt:1}
    $X_{\bt}$ has a delta-point;
  \item\label{thmbt:2}
    $X_{\bt}$ does not have Daugavet-points.
  \end{enumerate}
\end{thm}

Before giving a proof of the theorem
we will need some notation.
By definition, a \emph{tree} is a partially ordered set
$(\mathcal{T},\preceq)$ with the property that,
for every $t \in \mathcal{T}$, the set
$\{s \in \mathcal{T}: s \preceq t\}$
is well ordered by $\preceq$.
In any tree we use normal interval notation,
so that for instance a \emph{segment} is
$[s,t] = \{r \in \mathcal{T} : s \preceq r \preceq t\}$.
If a tree has only one minimal member, it is said to
be \emph{rooted} and the minimal member is called the
\emph{root} of the tree and is denoted $\emptyset$.
We have $\emptyset \preceq t$ for all $t \in \mathcal{T}$.
We say that $t$ is an \emph{immediate successor} of $s$
if $s \prec t$ and the set
$\{r \in \mathcal{T} : s \prec r \prec t\}$ is empty.
The set of immediate successors of $s$ we denote with $s^+$.
A sequence $B = \{t_n\}_{n=0}^\infty$ is a \emph{branch}
of $\mathcal{T}$ if $t_n \in \mathcal{T}$ for all $n$,
$t_0 = \emptyset$ and $t_{n+1} \in t_n^+$ for all $n \ge 0$.
If $s,t \in \bt$ are nodes such that neither
$s \preceq t$ nor $t \preceq s$, then
$s$ and $t$ are \emph{incomparable}.
An \emph{antichain} in a tree is a collection of elements
which are pairwise incomparable.

We consider the infinite binary tree,
$\bt = \bigcup_{n=0}^\infty \{0,1\}^n$,
that is, finite sequences of zeros and ones.
The order $\preceq$ on $\bt$ is defined
as follows:
If $s = \{s_1,s_2,\ldots,s_k\} \in \{0,1\}^k \subset \bt$ and
$t = \{t_1,t_2,\ldots,t_l\} \in \{0,1\}^l \subset \bt$, then
$s \preceq t$ if and only if $k \le l$
and $s_i = t_i$, $1 \le i \le k$.
As usual we denote with $|s|$ the cardinality of $s$,
i.e. $|s| = k$.
The \emph{concatenation} of $s$ and $t$ is
$s^\frown t = \{s_1,s_2,\ldots,s_k,t_1,t_2,\ldots,t_l\}
\in \{0,1\}^{k+l} \subset \bt$.
Clearly $s \preceq s^\frown t$ and
$s^+ = \{s^\frown 0, s^\frown 1\}$.
The infinite binary tree is rooted with $\emptyset = \{0,1\}^0$.

Following Talagrand \cite{MR554378,MR736065}
we say
that $\A \subseteq \mathcal{P}(\N)$
is an \emph{adequate} family if
\begin{itemize}
\item
  $\A$ contains the empty set and the singletons:
  $\{n\} \in \A$ for all $n \in \N$.
\item
  $\A$ is hereditary:
  If $A \in \A$ and
  $B \subseteq A$, then
  $B \in \A$.
\item
  $\A$ is compact with respect to the topology
  of pointwise convergence:
  Given $A \subset \N$,
  if every finite subset of $A$ is in $\A$,
  then $A \in \A$.
\end{itemize}
It is well-known that if we define a norm
on $c_{00}(\N)$ by
\begin{equation*}
  \|\sum_{i=1}^n a_i e_i\|
  = \sup_{A \in \A} |a_i|,
\end{equation*}
then the completion
$\ha = \overline{(c_{00}(\N),\|\cdot\|)}$
is a Banach space and $(e_i)_{i \in \N}$ is a $1$-unconditional
basis for $\ha$.
(See e.g. \cite[p.~410]{MR1226181}).
Since $\A$ is compact we get that for every
$x \in \ha$ there exists $A \in \A$ such
that $\|P_A x\| = \|x\|$.

There is a bijection between $\bt$ and $\N$
where the natural order on $\N$ corresponds to
the lexicographical order on $\bt$
(see \cite[p.~69]{MR2053392}).
The family $\A$ of all subsets of $\N$
corresponding to the branches of $\bt$ and their
subsets is an adequate family.
We get that $X_\bt := \ha$ is a Banach space
with $1$-unconditional basis $(e_t)_{t \in \bt}$.
Note that the span of the basis vectors corresponding
to any infinite antichain in $X_\bt$ is isometric to $c_0$,
and that the span of the basis vectors corresponding
to any branch in $X_\bt$ is isometric to $\ell_1$.

\begin{proof}[Proof of Theorem~\ref{thm:bt-har-delta-ikke-daugavet}~\ref{thmbt:1}]
  Consider
  \begin{equation*}
    x = \sum_{|t| > 0} 2^{-|t|} e_{t}.
  \end{equation*}
  Summing over branches we find that $\|x\| = 1$.
  We will show that $x$ is a delta-point.
  Define $z_{\emptyset} = 0$ and then for $t_0 \in \bt$
  \begin{equation*}
    z_{t_0^\frown 0} = z_{t_0} + e_{t_0^\frown 1}
    \quad\mbox{and}\quad
    z_{t_0^\frown 1} = z_{t_0} + e_{t_0^\frown 0}.
  \end{equation*}
  Here is a picture of $z_{(0,0)}$ and $z_{(0,1)}$:
  
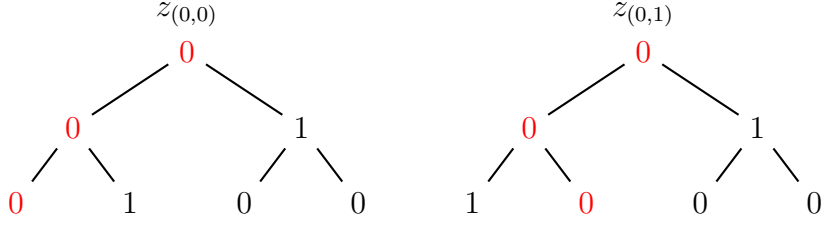
\begin{center}
    \begin{tikzpicture}
      \begin{scope}[shift={(-3,0)}]
        \draw (0,0.5) node {$z_{(0,0)}$};
        \draw (0,0) node (A) {${\color{red}0}$};
        \draw (-1.5,-1) node (B) {${\color{red}0}$};
        \draw (1.5,-1) node (C) {$1$};
        \draw (-2.25,-2) node (D) {${\color{red}0}$};
        \draw (-0.75,-2) node (E) {$1$};
        \draw (0.75,-2) node (F) {$0$};
        \draw (2.25,-2) node (G) {$0$};
        \draw[thick] (A)--(B);
        \draw[thick] (A)--(C);
        \draw[thick] (B)--(D);
        \draw[thick] (B)--(E);
        \draw[thick] (C)--(F);
        \draw[thick] (C)--(G);
      \end{scope}
      \begin{scope}[shift={(3,0)}]
        \draw (0,0.5) node {$z_{(0,1)}$};
        \draw (0,0) node (A) {${\color{red}0}$};
        \draw (-1.5,-1) node (B) {${\color{red}0}$};
        \draw (1.5,-1) node (C) {$1$};
        \draw (-2.25,-2) node (D) {$1$};
        \draw (-0.75,-2) node (E) {${\color{red}0}$};
        \draw (0.75,-2) node (F) {$0$};
        \draw (2.25,-2) node (G) {$0$};
        \draw[thick] (A)--(B);
        \draw[thick] (A)--(C);
        \draw[thick] (B)--(D);
        \draw[thick] (B)--(E);
        \draw[thick] (C)--(F);
        \draw[thick] (C)--(G);
      \end{scope}
    \end{tikzpicture}
  \end{center}
  From the definition it is clear that
  \begin{equation*}
    \frac{1}{2}\left(z_{t_0^\frown 0} + z_{t_0^\frown 1}\right)
    =
    z_{t_0}
    +
    \frac{1}{2}\left( e_{t_0^\frown 0} + e_{t_0^\frown 1} \right)
  \end{equation*}
  so by induction
  \begin{equation*}
    y_N := \frac{1}{2^N}\sum_{|t|=N} z_t
    = x - \sum_{|t|>N} 2^{-|t|}e_t.
  \end{equation*}
  Let $x^* \in S_{X_\bt^*}$ and $\delta > 0$ such
  that $x \in S(x^*,\delta)$. Find $N$ such that
  $x^*(y_N) > 1 - \delta$ which is possible
  since $\|\sum_{|t| > N} 2^{-|t|}e_t\| \to 0$
  as $N \to \infty$.
  But $x^*(y_N) > 1 - \delta$ means that
  there exists $t_0$ with $|t_0| = N$ such that
  $x^*(z_{t_0}) > 1 - \delta$.
  Let $E = (t_{i})_{i=1}^\infty$ be an infinite antichain
  of successors of $t_0$. Then $x^*(e_{t_i}) \to 0$
  as $i \to \infty$.
  Find $t_n$ such that
  \begin{equation*}
    x^*(z_{t_0} - e_{t_n}) > 1 - \delta.
  \end{equation*}
  By definition of $z_{t_0}$ we have
  $\{u \in \bt: u \preceq t_0\} \cap \supp(z_{t_0}) = \emptyset$
  hence $z_{t_0} - e_{t_n} \in S(x^*,\delta)$.
  Summing over a branch containing $t_n$ we get
  \begin{equation*}
    \|x - (z_{t_0} - e_{t_n})\|
    \ge
    \sum_{h=1,h \neq |t_n|}^\infty 2^{-h}
    + 2^{-|t_n|} + 1
    =
    2
  \end{equation*}
  as desired.
\end{proof}

Next is the proof that $X_\bt$
does not have Daugavet-points.
We first need a general lemma about
Daugavet-points.

Let $(e_i)_{i \in \N}$ be a 1-unconditional basis
in a Banach spaces $X$. Define
\begin{equation*}
  E_X = \{E \subset \N : \sum_{i \in E} e_i \in S_X\}.
\end{equation*}

\begin{lem}\label{lem:daug_antichain_E}
  Let $X$ be a Banach space with $1$-unconditional
  basis $(e_i)_{i \in \N}$.
  If $x \in S_X$ is a Daugavet-point,
  then $\|x - P_E x\| = 1$ for all $E \in E_X$.
\end{lem}

\begin{proof}
  Assume $x \in S_X \cap K_X$ and
  that there exists $\eta > 0$
  and $E \in E_X$ such that
  $\left\|x - P_E x\right\| < 1 - \eta$.

  Define $x^* = \frac{1}{|E|} \sum_{i\in E} e_i^* \in S_{X^*}$.
  Choose $\gamma > 0$ such that
  $\max_{i\in E} \frac{e_i^*(x)}{2} < 1 - \gamma$.
  If $y \in S(x^*, \frac{\gamma}{|E|})$,
  then it follows that
  $1 - \gamma < e^*_i(y)$ for all $i\in E$ and
  \begin{align*}
    \left\| x - y \right\|
    \leq
    \left\| x - P_E\frac{x}{2} \right\|
    +
    \left\| y - P_E\frac{x}{2} \right\|
    < 2 - \frac{\eta}{2},
  \end{align*}
  so $x$ is not a Daugavet-point.
\end{proof}

\begin{proof}[Proof of Theorem~\ref{thm:bt-har-delta-ikke-daugavet}~\ref{thmbt:2}]
  Assume $x \in S_{X_{\bt}} \cap K_{X_{\bt}}$.
  Let $E = \bigcup_{A \in \mx} A(1)$.
  From Lemma~\ref{lem:i-frste-c-infty}
  we see that $|E|$ is finite.
  Note that $E$ is an antichain.
  Indeed,
  assume $t_0, t_1 \in E$ with
  $t_0 \preceq t_1$ where $A(1) = \{t_1\}$ for some $A \in \mx$.
  Then since $x \in K_{X_{\bt}}$ and
  \begin{equation*}
    1 \ge \norm{P_{A \cup \{t_0\}}x}
    \ge
    \sum_{t \in A \cup \{t_0\}} e^*_t(x)
    =
    \sum_{t \in A \setminus \{t_0\}} e^*_t(x) + e^*_{t_0}(x)
    \ge \|P_A x\| = 1
  \end{equation*}
  we must have $t_0 = t_1$.

  We have $\|x - P_E x\| < 1$ by
  Lemma~\ref{lem:G_N-minker}~\ref{item:G_N-minker-1}.
  From Lemma~\ref{lem:daug_antichain_E} we get that
  $x$ is not a Daugavet-point since $E \in E_{X_\bt}$.
\end{proof}

Let us end this section with a remark about the
proof of Theorem~\ref{thm:bt-har-delta-ikke-daugavet}~\ref{thmbt:1}.
In order to prove that $X_\bt$ has a delta-point
we could have used dyadic trees.
Recall that a \emph{dyadic tree}
in a Banach space is a sequence
$(x_t)_{t \in \bt}$,
such that
$x_{t} = \frac{1}{2}(x_{t^{\frown} 0} + x_{t^{\frown} 1})$.

In fact, $x = \sum_{|t| > 0} 2^{-|t|} e_{t}$ is the root
of a dyadic tree. In order to show this one uses the
same $z_t$'s as in the above proof, but attach a copy
of $x$ to the node $t$.
Finally, we have the following result about dyadic
trees and delta-points.

\begin{prop}\label{prop:dyadic-tree-delta}
  If a Banach space $X$ contains a dyadic tree
  $(x_t)_{t \in \bt} \subset B_X$
  such that
  \begin{equation*}
    \limsup_{n \to \infty}
    (\min_{|t|=n} \{\|x_\emptyset - x_t\|\}) = 2,
  \end{equation*}
  then $x_{\emptyset}$ is a delta-point.
\end{prop}

\begin{proof}
  Let $\varepsilon > 0$ and find
  $n$ with $\|x_{\emptyset} - x_{t}\| \ge 2 - \varepsilon$
  for all $t$ with $|t|=n$.
  This means that $x_t \in \Delta_\varepsilon(x_{\emptyset})$.
  By definition of a dyadic tree
  \begin{equation*}
    x_{\emptyset} = \frac{1}{2^{n}} \sum_{|t| = n} x_{t},
  \end{equation*}
  so we have $x_{\emptyset} \in \conv \Delta_\varepsilon(x_{\emptyset})$.
\end{proof}

\section{A space with $1$-unconditional basis and daugavet-points}
\label{sec:modbinarytree}

In this section we will cut of the root
of the binary tree and modify the norm from
the example in the previous section to allow
the space to have Daugavet-points.

Let $\mbt = \bigcup_{n=1}^\infty \{0,1\}^n$
be the binary tree with the root removed.
Note that a branch $B = \{t_n\}_{n=1}^\infty$
in $\mbt$ corresponds to the branch
$\{t_n\}_{n=0}^\infty$ in $\bt$
where $t_0 = \emptyset$.

A \emph{$\lambda$-segment} in $\mbt$
is a set $S \subset \mbt$ of the form
$S = [s,t] \cup t^+$, where $[s,t]$ is a (possibly empty)
segment of $\mbt$. If $[s,t] = \emptyset$, then
$S = \{(0),(1)\}$.

Using the lexicographical order $\le$ on $\mbt$
we have a bijective correspondence to $\N$
with the natural order.
Let $\A$ be the adequate family
of subsets of $\N$
corresponding to subsets of branches and
subsets of $\lambda$-segments.
Using this adequate family we get a Banach
space $X_{\mbt} := \ha$ with $1$-unconditional basis
$(e_t)_{t \in \mbt}$.
We call $X_{\mbt}$ the modified binary tree space.
Note that $X_\mbt$ contains isometric copies
of $c_0$ and $\ell_1$ just like $X_\bt$.

As we saw in the proof of
Theorem~\ref{thm:bt-har-delta-ikke-daugavet}~\ref{thmbt:2}
the antichains in the tree play an important role
for the existence of Daugavet-points.

Define
\begin{equation*}
       \acF
       :=
       \{0\} \cup
       \{
       z \in S_{X_\mbt} : z(\mbt) \subset \{0,\pm 1\}
       \}.
\end{equation*}
The set $E_{X_{\mbt}}$ from Section~\ref{sec:binarytree}
can be described as the set of all
non-void finite antichains $E$ of $\mbt$
such that $|A \cap E| \le 1$ for all $A \in \A$.
Clearly $\supp(z) \in E_{X_{\mbt}}$ for every
$z \in \acF \setminus \{0\}$ and every $z$ with
$\supp(z) \in E_{X_{\mbt}}$ and $z(\mbt) \subset \{0,\pm 1\}$
belongs to $\acF$.
It is also clear that for every $E \in E_{X_{\mbt}}$
there exists a branch $B$ such that $B \cap E = \emptyset$.
We will see in Lemma \ref{Lem:konvekskombinasjonavF} and Theorem
\ref{thm:mbt_kar_daugpkt}  that the sets $E_{X_{\mbt}}$ and
$\mathfrak{F}$ will play an essential role in characterizing the
Daugavet-points of $X_\mbt$.

If $M$ is a finite subset of $\mbt$,
then we will use the notation
$K_M = \{\sum_{t \in M} a_t e_t : a_t \ge 0\}$ and
$\acF_M = \{z \in \acF : \supp(z) \subset M\}$.

First we prove a lemma which says that convex combinations
of elements in $\acF$ are dense in the unit ball
of $X_{\mbt}$.

\begin{lem}\label{Lem:konvekskombinasjonavF}
  Let $M$ be a finite subset of $\mbt$.
  Then
  \begin{equation*}
    \spann\left\{e_t : t\in M \right\} \cap B_{X_\mbt} = \conv \left(\acF_M\right)
  \end{equation*}
  that is, for every $x \in \spann\left\{e_t : t\in M \right\}  \cap B_{X_\mbt}$ we have
  \begin{equation}
    \label{eq:3}
    x = \sum_{k=1}^N \lambda_k z_k
  \end{equation}
  where $z_k \in \acF_M$, $\lambda_k > 0$,
  $\sum_{k=1}^N \lambda_k = 1$.
  In particular,
  $\ext(K_M \cap B_{X_\mbt}) = K_M \cap\acF_M$.
\end{lem}

\begin{proof}
  With $M_n$ denote the subset of $\mbt$ which
  corresponds to $\{1,\ldots,n\} \subset \N$.
  We will show, by induction, that for every $x\in K_{M_{2n}}\cap B_{X_\mbt}$ we have
  \[
  x= \sum_{k=1}^{N} \lambda_k z_k,
  \]
  where $z_k\in K_{\supp(x)}\cap \acF$, $\lambda_k>0$ and $\sum_{k=1}^{N} \lambda_k = 1$. As $K_M \subseteq K_{M_{2n}}$ for some $n\in\N$ and $z_k\in K_{\supp(x)}\cap \mathfrak{F}$, the result will follow.

  The base step is $x \in K_{M_2} \cap B_{X_{\mbt}}$
  with $e_t^*(x) \geq 0$ for $t \in M_2 = \{(0),(1)\}$.
  Write $e^*_{(0)}(x) = a_0$ and
  $e^*_{(1)}(x) = a_1$.
  Define $c = 1 - a_0 - a_1$,
  $z_{0} = e_{(0)}$, and
  $z_{1} = e_{(1)}$.
  Then
  \begin{equation*}
    x =
    \left(
      c \cdot 0 + a_0 z_0 + a_1 z_1
    \right)
  \end{equation*}
  is a convex combination of elements in $K_{\supp(x)}\cap\acF$.

  Assume the induction hypothesis holds for $n\in \N$.  Let $x \in K_{M_{2(n+1)}} \cap B_{X_\mbt}$. Let $t \in \mbt$ be the node such that $t^\frown 0$ corresponds to $2n+1$ and $t^\frown 1$
  to $2n+2$.
  Define
  \begin{equation*}
    x'
    =
    x
    - e^*_{t^\frown 0}(x) e_{t^\frown 0}
    - e^*_{t^\frown 1}(x) e_{t^\frown 1}.
  \end{equation*}
  By assumption we have $x' = \sum_{k=1}^{N} \lambda_k z_k$ with
  $\lambda_k>0$, $\sum_{k=1}^{N} \lambda_k = 1$ and
  $z_k\in K_{\supp(x')}\cap \acF$.

  Define the segment $A = \{ s \in \mbt : s \preceq t\}$ and the sets
  \begin{equation*}
    I = \left\{ k \in \left\{1,\ldots,N \right\} :
      P_A z_k = 0 \right\}
    \quad \mbox{and} \quad
    J = \{1,\ldots,N\} \setminus I.
  \end{equation*}
   For $k \in I$ we let
  \begin{equation*}
         z_{k,0} := z_k + e_{t^\frown 0}
         \quad \mbox{and} \quad
         z_{k,1} := z_k + e_{t^\frown 1}.
  \end{equation*}
  Since $z_k \in K_{\supp(x')} \cap \acF$
  we get $z_{k,0}$, $z_{k,1} \in K_{\supp(x)} \cap \acF$ and
  \begin{equation*}
    \sum_{s \in A} e^*_{s}(x') = \sum_{s \in A} e^*_{s}(x)
    = \sum_{k \in J} \lambda_k.
  \end{equation*}
  Thus, by definition of the norm we have,
  \begin{equation*}
       0 \leq e^*_{t^\frown 0}(x) + e^*_{t^\frown 1}(x)
       \leq 1 - \sum_{s \in A}  e^*_{s}(x) = \sum_{k\in I}\lambda_k.
  \end{equation*}
  Write
  $e^*_{t^\frown 0}(x) = a_0$ and
  $e^*_{t^\frown 1}(x) = a_1$.
  Define
  $c = \sum_{k \in I} \lambda_k - a_0 - a_1$.
  Let $m = \frac{1}{\sum_{k\in I} \lambda_k  } $.
  It follows that
  \begin{align*}
    x
    &= x' + a_0e_{t^\frown 0} +a_1 e_{t^\frown 1}\\
    &=
    \sum_{k \in J} \lambda_k z_k
    +
    \sum_{k \in I} \lambda_k z_k
    +
    \sum_{k \in I} \lambda_k
    \left(\frac{a_0}{m} e_{t^\frown 0}
    + \frac{a_1}{m} e_{t^\frown 1}\right)\\
    &=\sum_{k \in J} \lambda_k z_k
    +
    \sum_{k \in I} \lambda_k \frac{(a_0+a_1+c)}{m} z_k
    +
    \sum_{k \in I} \lambda_k
    \left(\frac{a_0}{m} e_{t^\frown 0}
    + \frac{a_1}{m} e_{t^\frown 1}\right)\\
    &=
    \sum_{k \in J} \lambda_k z_k
    +
    \sum_{k \in I} \lambda_k
    \left(\frac{a_0}{m} z_{k,0}
    + \frac{a_1}{m} z_{k,1}
    + \frac{c}{m} z_k
    \right)
  \end{align*}
  which is a convex combination of elements in
  $K_{\supp(x)} \cap \acF$.
\end{proof}

With the above lemma in hand we are able to
characterize Daugavet-points in $X_{\mbt}$
in terms of $E_{X_{\mbt}}$.
This will give us an easy way to identify
and give examples of Daugavet-points.

\begin{thm}\label{thm:mbt_kar_daugpkt}
  Let $x \in S_{X_{\mbt}}$, then the following are equivalent
  \begin{enumerate}
  \item\label{item:thm-mbt-1}
    $x$ is a Daugavet-point;
  \item\label{item:thm-mbt-2}
    $\left\| x - P_E x \right\| = 1$,
    for all $E\in E_{X_{\mbt}}$;
  \item\label{item:thm-mbt-3}
    for any $z\in \acF$, either
    $\left\| x - z \right\| = 2$
    or
    for all $\varepsilon > 0$ there exists $s \in \mbt$
    such that $z \pm e_s \in \acF$
    and $\left\| x - z \pm e_s \right\| > 2 - \varepsilon$.
  \end{enumerate}
\end{thm}

\begin{proof}
  As usual we will assume that $x \in K_{X_\mbt}$ throughout.

  \ref{item:thm-mbt-1} $\Rightarrow$ \ref{item:thm-mbt-2}
  is Lemma~\ref{lem:daug_antichain_E}.

  \ref{item:thm-mbt-2} $\Rightarrow$ \ref{item:thm-mbt-3}.
  Let $\varepsilon > 0$,
  $z\in \acF$ and $E = \supp(z)$.
  We have assumed that $\|x - P_E x\| = 1$.

  By definition of $M(x-P_E x)$ we have
  $A \cap E = \emptyset$ for every $A \in M(x-P_Ex)$.
  If there, for some $A \in M(x-P_Ex)$,
  exists $t \in E$ and $s_0 \in A$ such that $t \preceq s_0$,
  or $t \in E$ such that $s \preceq t$ for all $s \in A$,
  then we are done since $e^*_t(x)= 0$ and
  \begin{equation*}
    \|x - z\| \ge \sum_{s \in A} |e^*_s(x)| + |e_t^*(z)| = 2.
  \end{equation*}
  So from now on we assume that no such $A$ exists.

  Assume that there exists $A \in M(x-P_Ex)$
  that is a subset of a branch $B$.
  By definition of the norm, we have $e^*_t(x) = 0$
  for $t\in B\setminus A$, and by the assumption above,
  we also have $B\cap E =\emptyset$.
  Since $|e^*_t(x)| \to 0$ as $|t| \to \infty$ for $t \in B$
  we can find $s \in B$ with $|e^*_s(x)| < \varepsilon/2$
  and hence
  \begin{equation*}
    \|x - z \pm e_s\|
    \ge \sum_{t \in A, t \neq s} |e^*_t(x)| + |e^*_s(x) \pm 1|
    \ge 2 - \varepsilon.
  \end{equation*}
  This concludes the case where $A$ is a subset of a branch.

  Suppose for contradiction that
  no $A \in M(x - P_E x)$ is a subset of a branch,
  then every $B \in M(x - P_E x)$ is a subset of
  a $\lambda$-segment.
  By Lemma~\ref{lem:k_nendelig} we must have
  $|M( x - P_E x )| < \infty$.

  Choose any $B \in M(x - P_E x)$ and write
  \begin{equation*}
    B =
    \{b_1 \prec b_2 \prec \cdots \prec b_n\}
    \cup \{b^\frown 0, b^\frown 1\},
  \end{equation*}
  where $b_n \preceq b$.
  In particular $e^*_s(x) \neq 0$ for $s \in b^+$.

  Let
  $R =
  \left\{ t \in E :
    b^\frown 0 \prec t \right\}$
  and $E_1 =
  \left(E \cup \left\{b^\frown 0 \right\} \right)\setminus R$.
  From the assumptions above
  $E \cap \{t :  t \preceq b^\frown 0\} = \emptyset$,
  so $E_1 \in E_{X_{\mbt}}$.

  Let $C \in M(x - P_{E_1}x)$. Notice that $C\cap \left\{t : b^\frown 0 \preceq t \right\}= \emptyset$. Otherwise, by definition of the norm, we get the contradiction
  \begin{equation*}
    1 = \|P_{C\cap\left\{b^\frown 0 \right\}}x\|
    = \sum_{t \in C} |e^*_t(x)| + |e^*_{b^\frown 0}(x)|
    > \sum_{t \in C} |e^*_t(x)|
    = \|P_C x\| = 1.
  \end{equation*}
  Hence
  $P_C(x - P_{E_1}x) = P_C(x - P_E x)$
  and $C \in M(x - P_E x)$.

  We have $M(x - P_{E_1}x) \subseteq M(x - P_E x)$,
  but since $B \cap {E_1} \neq \emptyset$
  we have $B \notin M(x - P_{E_1}x)$
  so the inclusion is strict.

  We now have $|M(x - P_{E_1}x)| < |M(x - P_E x)|$
  and no $C \in M(x - P_{E_1}x)$ is a subset of
  a branch. We can use the argument above
  a finite number of times until we are
  left with $E_m \in E_{X_{\mbt}}$
  with $\|x - P_{E_m}x\| = 1$
  and $M(x - P_{E_m}x) = \emptyset$
  which contradicts Lemma~\ref{lem:overkill}.

  Finally, \ref{item:thm-mbt-3} $\Rightarrow$ \ref{item:thm-mbt-1}.
  Choose $\varepsilon > 0$.
  Let $y \in B_{X_\mbt}$ with finite support.
  Then by Lemma~\ref{Lem:konvekskombinasjonavF},
  we can write $y = \sum_{k=1}^{n} \lambda_k z_k$,
  with $z_k \in \acF$, $\lambda_k \ge 0$
  and $\sum_{k=1}^n \lambda_k = 1$.
  Let $D_1 = \left\{k \in \left \{1,\ldots, n \right\} :
    \left\| x - z_k \right\| = 2 \right\}$ and
  $D_2 = \left\{1,\ldots, n \right\} \setminus D_1$.
  We can, by assumption, for each $k \in D_2$
  find $s_k \in \mbt$ such that $z_k \pm e_{s_k} \in \acF$
  with $\left\| x - z_k \pm e_{s_k} \right\| > 2 - \varepsilon$.
  Then $y \in \conv \Delta_\varepsilon(x)$ since
  \begin{equation*}
    y
    =
    \sum_{i\in D_1} \lambda_k z_k
    + \sum_{k\in D_2}^{} \frac{\lambda_k}{2} (z_k + e_{s_k})
    + \sum_{k\in D_2}^{} \frac{\lambda_k}{2} (z_k - e_{s_k}).
  \end{equation*}
  The set of all such $y$ is dense in $B_{X_\mbt}$,
  hence $B_{X_\mbt} = \clonv \Delta_\varepsilon(x)$
  so $x$ is a Daugavet-point.
\end{proof}

\begin{cor}\label{cor:konkret-betingelse-for-daugavetpkt}
  If $x \in S_{X_{\mbt}}$ such that $\|P_A x\| = 1$
  for all branches $A$, then $x$ is a Daugavet-point.
\end{cor}

\begin{proof}
  Let $E \in E_{X_{\mbt}}$.
  There exists a branch $B$ such that $B \cap E = \emptyset$.
  Then $\|x - P_E x\| \ge \|P_B x\| = 1$.
  By Theorem~\ref{thm:mbt_kar_daugpkt} $x$ is a Daugavet-point.
\end{proof}

With a characterization of Daugavet-points in hand
we can now prove the main result of this section.

\begin{thm}\label{thm:MBT_have_daug_delta}
  In $X_{\mbt}$ we have that
  \begin{enumerate}
  \item
    there exists
    $x \in S_{X_{\mbt}}$ which is a Daugavet-point;
  \item
    there exists
    $w \in S_{X_{\mbt}}$ which is a delta-point,
    but not a Daugavet-point.
  \end{enumerate}
\end{thm}

\begin{proof}
  Let $x = \sum_{t \in \mbt} 2^{-|t|}e_{t}$.
  We have that $x$ is a Daugavet-point
  by Corollary~\ref{cor:konkret-betingelse-for-daugavetpkt}.

  The next part of the proof is similar to the proof of
  Theorem~\ref{thm:bt-har-delta-ikke-daugavet}~\ref{thmbt:1}.
  We will show that a shifted version of $x$ is a delta-point
  which is not a Daugavet-point.
  Define an operator on the modified binary tree:
  \begin{equation*}
    L\left(
      \sum_{|t| > 0}
      a_{t} e_{t}
    \right)
    =
    \sum_{|t| \ge 0}
    a_{0^\frown t} e_{0^\frown t}
    +
    \sum_{|t| \ge 0}
    a_{1^\frown t} e_{(1,0)^\frown t},
  \end{equation*}
  where $t = \emptyset$ when $|t|=0$.

  Define $w = L(x)$.
  Let $x^* \in S_{X^*_\mbt}$ and $\delta > 0$
  such that $w \in S(x^*,\delta)$.
  Just as in the proof of
  Theorem~\ref{thm:bt-har-delta-ikke-daugavet}~\ref{thmbt:1}
  we can find $z_{t_0} \in S_{X_\mbt}$ whose support is
  an antichain (i.e. $z_{t_0} \in \acF$) and
  we can find $e_{t_n}$ such that
  $z_{t_0} - e_{t_n} \in S(x^*,\delta)$.
  Summing over a branch containing $t_n$ we
  get $\|w - (z_{t_0} - e_{t_n})\| = 2$.

  Let $E = \{(0),(1,0)\}$.
  Then $\|w - P_E w\|
  = \sum_{i=2}^\infty 2^{-i} = \frac{1}{2} < 1$
  so by Theorem~\ref{thm:mbt_kar_daugpkt}
  $w$ is not a Daugavet-point.
\end{proof}

In \cite{AHLP}, the property that the unit ball
of a Banach space is the closed convex hull
of its delta-points was studied.
We will next show that $X_\mbt$ satisfies
something much stronger, the unit ball
is the closed convex hull of a subset
of its Daugavet-points.

If $D$ is the set of all Daugavet-points in $X_{\mbt}$ define
\begin{equation*}
  D_B =
  \{x \in D: \norm{P_Bx} = 1
  \ \text{for all branches}\ B\ \text{of}\ \mbt\}.
\end{equation*}
The proof of Theorem~\ref{thm:MBT_have_daug_delta}
shows that $D_B$ is non-empty.

For $t_0 \in \mbt$,
let $S_{t_0}$ be the shift operator on $X_{\mbt}$
that shifts the root to $t_0$, that is
\begin{equation}\label{eq:defn_shift_op}
  S_{t_0}(\sum_{t \in \mbt} a_{t} e_{t})
  =
  \sum_{t \in \mbt} a_{t} e_{t_0^\frown t}
\end{equation}
It is clear that $S_{t_0}$ is an isometry on $X_{\mbt}$.

\begin{prop}\label{prop:mbt_har_convDLD2P}
  The space $X_{\mbt}$ satisfies $B_{X_\mbt} = \clonv\,(D_B)$.
\end{prop}

\begin{proof}
  Let $y \in B_{X_\mbt}$.
  We may assume that $y$ has finite support,
  since such $y$ are dense in $B_{X_\mbt}$.
  By Lemma~\ref{Lem:konvekskombinasjonavF},
  we can write $y = \sum_{k = 1}^n \lambda_k z_k$
  where $z_k \in \acF$, $\lambda_k \ge 0$
  and $\sum_{k=1}^n \lambda_k = 1$.

  Fix $z \in \acF$. Let $m := \max\{|t| : t \in \supp(z)\}$.
  \begin{align*}
    \mathcal{B}
    = \{t \in \mbt: |t| = m,
    \sum_{s \preceq t} |e^*_s(z)|  = 0\}.
  \end{align*}
  Choose any $x_0 \in D_B$ and
  use the shift operator in \eqref{eq:defn_shift_op}
  to define
  \begin{align*}
    x :=
    \sum_{t \in \mathcal{B}}
    S_t(x_0).
  \end{align*}
  Observe that $z \pm x$ takes its norm along
  every branch, so by
  Corollary~\ref{cor:konkret-betingelse-for-daugavetpkt}
  both $z \pm x \in D_B$.

  Repeat this construction for $z_k$ to create $x_k$
  for $k \in \{1,\ldots,n\}$.
  Then
  \begin{align*}
    y =
    \sum_{k = 1}^n \frac{\lambda_k}{2} (z_k + x_k)
    +
    \sum_{k = 1}^n \frac{\lambda_k}{2} (z_k - x_k),
  \end{align*}
  is a convex combination of Daugavet-points in $D_B$.
\end{proof}

Our next result is that $X_\mbt$ has the remarkable
property that the Daugavet-points are weakly
dense in the unit ball. So in a sense there
are lots of Daugavet-points, but of course
not enough of them in order for $X_\mbt$
to have the Daugavet property.
First we need a lemma.
For $t \in \mbt$, $S_t$ denotes
the shift operator defined in \eqref{eq:defn_shift_op} above.

\begin{lem}\label{lem:funtionals-on-ha}
  Let $x^*\in S_{X_\mbt^*}$ and $s\in \bt$.
  For any $x\in S_{X_\mbt}$ and
  $\varepsilon>0$ there exist some infinite antichain
  $E = \{t_i\}_{i=1}^\infty$ with the
  following properties
  \begin{enumerate}
  \item\label{item:func-ha1}
    $\left\|\sum_{i=1}^n e_{t_i}\right\| = 1$ for all
    $n \in \N$;
  \item\label{item:func-ha2}
    $s \preceq t$ for all $t \in E$;
  \item\label{item:func-ha3}
    $\left|x^*(S_t x) \right| < \varepsilon$ for all $t\in E$.
  \end{enumerate}
\end{lem}

\begin{proof}
  Pick any $x^*\in S_{X^*_\mbt}$, $s\in \bt$ and $x\in S_{X_{\mbt}}$.
  It is not difficult to find an infinite antichain
  $E = \{ t_i \}_{i=1}^\infty$
  satisfying
  \ref{item:func-ha1} and \ref{item:func-ha2}.
  Since $E$ is an antichain we have
  $\|\sum_{i=1}^n S_{t_i}(x)\| = 1$
  for all $n \in \N$.
  Hence
  \begin{equation*}
    \lim_{i \rightarrow \infty}x^*\left(S_{t_i}x \right) = 0,
  \end{equation*}
  and then we can find $n \in \N$ such that
  $\left|x^*\left(S_{t_i}x \right) \right| < \varepsilon$
  for all $i \geq n$.
  Now $E' = E \setminus \left\{t_i \right\}_{i=1}^n$
  satisfies \ref{item:func-ha1}, \ref{item:func-ha2}
  and \ref{item:func-ha3}.
\end{proof}

\begin{thm}
  \label{thm:mbt-daugpoints-weakly-dense}
  In $X_{\mbt}$
  every non-empty relatively weakly open subset of $B_{X_{\mbt}}$
  contains a Daugavet-point.
\end{thm}

\begin{proof}
  Since vectors with finite support are norm dense in $B_{X_{\mbt}}$,
  it enough show that for any $y \in B_{X_{\mbt}}$ with finite support
  and any relatively weakly open neighbourhood of $y$ of the form
  \[
   W := \{x \in B_{X_{\mbt}}: |x_i^*(y - x)| < \eps, i = 1, \ldots, n\},
  \]
  where $x_i^* \in S_{X^*_\mbt}, i = 1, \ldots, n$ and $\eps > 0$,
  contains a Daugavet-point.

  Let $m:= \max\{|t|: t \in \supp(y)\},$
  and for $t \in \mbt$ with $|t| = m$ define
  \begin{equation*}
    \mu_t := 1 - \sum_{s \preceq t} |e^*_s(y)|
  \end{equation*}
  and
  \begin{equation*}
    \mathcal{N} := \{t \in \mbt : |t| = m, \mu_t > 0\}.
  \end{equation*}
  From Corollary~\ref{cor:konkret-betingelse-for-daugavetpkt}
  we have that $g = \sum_{s \in \mbt} 2^{-|s|} e_s$
  is a Daugavet-point.
  By Lemma~\ref{lem:funtionals-on-ha}
  for each $t \in \mathcal{N}$ there exists
  $t \preceq b_t$ such that
  $|x_i^*(S_{b_t}g)| < \eps/2^m$ for $i = 1, \ldots, n$.
  Now put
  \[
    x = y + \sum_{t \in \mathcal{N}} \mu_t S_{b_t}(g).
  \]
  By construction $x \in S_{X_{\mbt}}$ and we have $x \in W$ since
  \begin{align*}
    |x_i^*(y - x)|
    &=\left| x_i^*
      \left(
        \sum_{t \in \mathcal{N}} \mu_t S_{b_t}(g)
      \right)
    \right|
    % \\ &
    \le \sum_{t \in \mathcal{N}} \mu_t |x_i^* (S_{b_k} g)|
    < \frac{\eps}{2^m} \sum_{t \in \mathcal{N}} \mu_t < \eps.
  \end{align*}
  Using Theorem~\ref{thm:mbt_kar_daugpkt}
  we will show that $x$ is a Daugavet-point.
  Indeed, let $E \in E_{X_{\mbt}}$.
  Then there exists a branch $A$ with $A \cap E = \emptyset$.
  Let $t \in A$ with $|t| = m$.
  If $t \notin \mathcal{N}$, then
  \begin{align*}
    \norm{x - P_Ex}
    \ge \sum_{s \preceq t} |e^*_s(y)| = 1.
  \end{align*}
  If $t \in \mathcal{N}$,
  then since $S_{b_{t}}(g)$ is a Daugavet-point,
  there exists a branch $B$ with $t \in B$ such that
  $\norm{S_{b_t}(g) - P_ES_{b_t}(g)}
  = \sum_{s \in B}|S_{b_t}(g)_s| = 1$.
  Thus
  \begin{align*}
    \norm{x - P_Ex}
    &\ge
    \sum_{s \preceq t} |e^*_s(y)|
    +
    \sum_{\substack{s \in B,\\ s \succ b_t}} \mu_t|S_{b_t}(g)_s|
    = 1 - \mu_{t} + \mu_{t} = 1,
  \end{align*}
  and we are done.
\end{proof}

\begin{quest}
  How ``massive'' does the set of Daugavet-points in $S_X$
  have to be in order to ensure that a Banach space $X$ fails to
  have an unconditional basis?
\end{quest}

If $S$ is a slice of the unit ball of $X_{\mbt}$,
then the above proposition tells us that
$S$ contains a Daugavet-point $x$.
Then by definition of Daugavet-points there
exists for any $\varepsilon > 0$ a $y \in S$
with $\|x-y\| \ge 2 - \varepsilon$.
Thus the diameter of every slice of the unit ball
of $X_{\mbt}$ is $2$, that is $X_{\mbt}$
has the \emph{local diameter two property}.

The next natural question is whether the diameter of every non-empty
relatively weakly open neighborhood in $B_{X_\mbt}$
equals $2$, that is, does $X_\mbt$ have
the \emph{diameter two property}?
The answer is no, in fact,
every Daugavet-point in $D_B$ has a weak
neighborhood of arbitrary small diameter.
Let us remark that the first example of
a Banach space with the local diameter two property,
but failing the diameter two property
was given in \cite{BGLPRZ4}.

\begin{prop}
  In $X_{\mbt}$ every $x \in D_B$ is a point of
  weak- to norm-continuity for the identity map on $X_{\mbt}$.
  In particular, $X_\mbt$ fails the diameter two property.
\end{prop}

\begin{proof}
  Let $\eps > 0$ and $x \in D_B$.
  Let $n \in \N$ be such that $\norm{\sum_{|t|>n} x_t e_t} < \frac{\eps}{8}$.
  Consider the weak neighborhood $W$ of $x$
  \begin{align*}
    W =
    \{y \in B_{X_{\mbt}} :
    |e_{t}^*(x  - y)| < \frac{\eps}{2^{|t| + 3}}, |t| \le n\}.
  \end{align*}
  We want to show that the diameter of $W$ is less than $\eps$.
  Let $y = \sum_{t \in \mbt} y_t e_t \in W$.
  Let $A$ be a subset of a branch or of
  a $\lambda$-segment in $\mbt$.
  Since $|x_{t} - y_{t}| < \eps 2^{-|t| -3}$ for $|t| \le n$,
  $\norm{\sum_{|t|>n} x_t e_t} < \frac{\varepsilon}{8}$, and
  $x$ attains its norm along every branch of $\mbt$,
  we have
  \begin{equation*}
    \sum_{\substack{t \in A\\ |t| \le n}} |y_{t}|
    >
    \sum_{\substack{t \in A\\ |t| \le n}} |x_{t}|-|x_t-y_t|
    >
    \sum_{\substack{t \in A\\ |t| \le n}} |x_{t}|
    -
    \frac{\eps}{8}
    >
    1 - \frac{\eps}{4}.
  \end{equation*}
  Hence
  $\sum_{\substack{t \in A\\ |t| > n}} |y_{t}|
  < \frac{\eps}{4}$,
  and thus
  \begin{align*}
    \sum_{t \in A}|x_{t} - y_{t}|
    &=
    \sum_{\substack {t \in A\\ |t| \le n}}|x_{t}- y_{t}|
    +
    \sum_{\substack {t \in A\\ |t| > n}}|x_{t}- y_{t}|
    \\
    &<
    \sum_{\substack {t \in A\\ |t| \le n}}\eps2^{-|t|-3}
    +
    \sum_{\substack {t \in A\\ |t| > n}} |x_{t}|
    +
    \sum_{\substack {t \in A\\ |t| > n}}| y_{t}|
    \\
    &<
    \frac{\eps}{8}
    +
    \frac{\eps}{8}
    +
    \frac{\eps}{4}
    = \frac{\eps}{2}.
  \end{align*}
  From this it follows that
  the diameter of $W$ is less than $\varepsilon$.
\end{proof}

Recall from \cite{ALL}
that a Banach space $X$ is \emph{locally almost square}
if for every $x \in S_X$ and $\varepsilon > 0$
there exists $y \in S_X$ such that
$\|x \pm y\| \le 1 + \varepsilon$.

It is known that every locally almost square Banach
space $X$ has the local diameter two property.
As noted above $X_{\mbt}$ has the local diameter
two property, but it is not locally almost square
as the following proposition shows.

\begin{prop}
  $X_{\mbt}$ is not locally almost square.
\end{prop}

\begin{proof}
  Consider $x = \frac{1}{4}e_{(0)} + \frac{3}{4}e_{(1)}$.
  Let $0 < \eps < \frac{1}{4}$ and suppose there exists
  $y = \sum_{t \in \mbt} y_t e_t \in S_{X_{\mbt}}$
  with $\norm{x \pm y} \le 1 + \eps < \frac{5}{4}$.
  Then clearly $|y_{(1)}| \le \frac{1}{4} +\eps$.
  By considering $-y$ if necessary we may assume that $y_{(1)} \ge 0$.
  Then
  \begin{align*}
    1 + \eps
    &\ge
    \max_{\pm}\{|\frac{1}{4} \pm y_{(0)}|
    + |\frac{3}{4} \pm y_{(1)}|\}\\
    & \ge |\frac{1}{4} - y_{(0)}| + \frac{3}{4} + |y_{(1)}|\\
    & \ge |y_{(0)}| - \frac{1}{4} + \frac{3}{4} + |y_{(1)}|,
  \end{align*}
  which yields $|y_{(0)}| + |y_{(1)}| \le \frac{1}{2} + \eps < \frac{3}{4}$.
  Thus since $\norm{y} = 1$ there must exist
  a subset $A$ of a branch or a $\lambda$-segment
  such that $|A \cap \{(0),(1)\}| = 1$
  and $\sum_{t \in A} |y_t| = 1$.
  Let $s \in A \cap \{(0),(1)\}$.
  \begin{align*}
    \frac{5}{4}
    > \norm{x \pm y}
    &=
    \max_{\pm} | x_s \pm y_s |
    +
    \sum_{\substack{t \in A \\ t \neq s}} |y_t|
    = |x_s| + |y_s| + 1 - |y_s|
  \end{align*}
  and we get the contradiction $|x_s| < \frac{1}{4}$.
\end{proof}

Recall from \cite{HLP} that a Banach space $X$
is \emph{locally octahedral} if for every $x \in S_X$
and $\varepsilon > 0$, there exists $y \in S_X$
such that $\|x \pm y\| \ge 2 - \varepsilon$.

It is known that every Banach space with
the Daugavet property is octahedral.
Even though the modified binary tree space
have lots of Daugavet-points, as seen
in Proposition~\ref{prop:mbt_har_convDLD2P},
it is not even locally octahedral.

\begin{prop}
  $X_{\mbt}$ is not locally octahedral.
\end{prop}

\begin{proof}
  Consider $x = \frac{1}{2}(e_{(0)} + e_{(1)}) \in S_{X_{\mbt}}$.
  We want to show that for all $y \in S_{X_{\mbt}}$ we
  have $\min\|x \pm y\| \le \frac{3}{2}$.

  Let $y = \sum_{t \in \mbt} y_t e_t \in S_{X_{\mbt}}$.
  Let $A$ be a subset of a branch or a $\lambda$-segment.
  If $A \neq \{(0),(1)\}$, then
  \begin{equation*}
    \sum_{t \in A} |x_t \pm y_t|
    \le
    \begin{cases}
      \frac{1}{2} + \sum_{t \in A}|y_t|;
      & A \cap \{{(0)},{(1)}\} \neq \emptyset \\
      \sum_{t \in A}|y_t|;
      & A \cap \{(0),(1)\} = \emptyset
    \end{cases}
    \le
    \begin{cases}
      \frac{3}{2} \\
      1
    \end{cases}
  \end{equation*}
  If $A = \{(0),(1)\}$, then,
  since $|y_{(0)}| + |y_{(1)}| \le 1$
  and a convex function
  attains its maximum at the extreme points,
  we get
  \begin{equation*}
    |\frac{1}{2} + y_{(0)}|
    +
    |\frac{1}{2} + y_{(1)}|
    +
    |\frac{1}{2} - y_{(0)}|
    +
    |\frac{1}{2} - y_{(1)}|
    \le 3.
  \end{equation*}
  Hence $\min\|x \pm y\| \le \frac{3}{2}$.
\end{proof}

\def\cprime{$'$} \def\cprime{$'$} \def\cprime{$'$}
\providecommand{\bysame}{\leavevmode\hbox to3em{\hrulefill}\thinspace}
\providecommand{\MR}{\relax\ifhmode\unskip\space\fi MR }
\providecommand{\MRhref}[2]{%
  \href{http://www.ams.org/mathscinet-getitem?mr=#1}{#2}
}
\providecommand{\href}[2]{#2}

\end{document}